\documentclass[11pt]{article}
\usepackage{amsmath}
\usepackage{amssymb}
\usepackage{amscd}
\usepackage{amsmath}
\usepackage[all]{xy}
\usepackage{setspace}
\usepackage{graphics,epsfig}
\usepackage{graphicx}
\usepackage{float}
\usepackage{stackrel}
\usepackage{tikz-cd}
\usepackage{tikz}
\usetikzlibrary{positioning}

\DeclareMathAlphabet{\eusm}{OT1}{eusm}{m}{n}
\newtheorem{thm}{Theorem}[section]
\newtheorem{prop}[thm]{Proposition}
\newtheorem{cor}[thm]{Corollary}
\newtheorem{exam}[thm]{Example}
\newtheorem{rem}[thm]{Remark}
\newtheorem{lem}[thm]{Lemma}

\def\vsp{\vspace{1ex}}

\newenvironment{proof}{\par\noindent{\bf Proof \,}}{$\hfill
\Box$\par\bigskip} \textheight 23cm
\begin{document}
\begin{center}
{\rm {\LARGE Isomorphism Problem For Uniserial Modules Over An Arbitrary Ring}}
\end{center}
\vsp

\begin{center}
{\rm {\Large Gabriella D$^{'}$Este$^a$, Fatma Kaynarca$^b$ and Derya Keskin
T\"{u}t\"{u}nc\"{u}$^c$
}}\\
${}^a$ Department of Mathematics, University of Milano,\\
Milano, Italy\\
e-mail: gabriella.deste@unimi.it\\
${}^b$ Department of Mathematics, University of Afyon Kocatepe,\\
03200 Afyonkarahisar, Turkey\\
e-mail: fkaynarca@aku.edu.tr \\
${}^c$ Department of Mathematics, University of Hacettepe,\\
06800 Beytepe, Ankara, Turkey\\
e-mail: keskin@hacettepe.edu.tr\\
\end{center}

\vsp
\begin{center}
{\large {\bf Abstract}} \end{center}

Firstly, we give a partial solution to the isomorphism problem for uniserial modules of finite length with the help of the morphisms between these modules over an arbitrary ring. Later, under suitable assumptions on the lattice of the submodules, we give a method to partially solve the isomorphism problem for uniserial modules over an arbitrary ring. Particular attention is given to the natural class of uniserial modules defined over algebras given by quivers.

\footnote[0]{$2010$ Mathematics Subject Classification{\rm :} Primary 16 D 10 ; Secondary 16 G 20 \\
\hspace{5mm} Key Words{\rm:} Uniserial modules, uniform modules, quivers and representations }

\section{Introduction}
Giving a method for deciding when two uniserial modules over an artin algebra are isomorphic is an open problem which has been asked in \cite[p. 411]{ARS}.
As a partial answer to this question in 1997, Bongartz proves a result, which gives an intrinsic inductive characterization of some algebras having only finitely many uniserial modules up to isomorphism (see \cite{Bongartz}). In 1998, Huisgen-Zimmermann gives a solution to this problem over
finite dimensional algebras over algebraically closed fields in \cite{birge}. In this direction \cite{BZ} and its references have important results based on algebraic geometry. In 2003, Mojiri characterizes isomorphism classes of uniserial modules over a biserial algebra in his thesis (see \cite{Mojiri}). Later, Boldt and Mojiri continue to work on this problem in 2008 (see \cite{B-M}). We should also note that in 2006, it is proven by P\v{r}\'{i}hoda that for two uniserial modules $U$ and $V$ over any ring $R$,
$U\cong V$ if and only if  there is a monomorphism $f: U\to V$ and an epimorphism $g: U\to V$ (\cite[Remark 2.1]{P}). We refer to \cite{BW} for very interesting and recent results
on computer algebra concerning the so-called ``Module Isomorphism Problem"  and many related
isomorphism problems for algebraic structures. Note that there exists an artin algebra having two non-isomorphic
uniserial left modules of length two with the same socle and top. For example, let $K$ be an arbitrary
field. If $A=K\Gamma$, where $\Gamma$ is the Kronecker quiver
$1\stackrel[\beta]{\alpha}{\rightrightarrows}2$, we have
$U_{k}\ncong U_{l}$ for $k\neq l$ for any uniserial module
$U_{k}=A e_{1}/A(\beta-k\alpha)$ for $k\in K$ (see
\cite{ARS} or \cite{birge}).

Inspired and motivated by above question and works, we give our first main result for a partial solution to this question over any ring (not only artin algebras) as Theorem~\ref{res1} and Corollary~\ref{res2}: Let $L$ and $M$ be uniform (for example uniserial) modules of finite length $n$ over an arbitrary ring. The following are
equivalent:
\begin{enumerate}
\item[(i)] $L$ is isomorphic to $M$;
\item[(ii)] There exists two morphisms $f: L\to M$, $g: M\to L$ and a nonzero element $x\in L$ such that $(g\circ f)(x)=x$;
\item[(iii)] There exists $n$ morphisms $f_1, \ldots , f_n$ with $f_1:L\to M$, $f_2: M\to L, \ldots $ such that $f_n\circ \ldots \circ f_1\neq 0$.
\end{enumerate}

Later, we illustrate that we cannot replace (iii) by a similar
condition on $n-1$ maps even if the two uniserial modules are projective-injective (Example~\ref{ex1}).

The second main result of this paper, Theorem~\ref{fixpoint}, is a kind of ``two fixed points theorem": Let $L$ and $M$ be uniserial modules with the following property:\\

(*)~~The lattices of the submodules of $L$ and $M$ are isomorphic to the same\\
 \hspace*{0.5in} sublattice of $\mathbb{Z}\cup\{+\infty, -\infty\}$.\\

Then the following conditions are equivalent:
\begin{enumerate}
  \item [(i)] $L$ is isomorphic to $M$;
  \item [(ii)] There is an endomorphism of $L$ which factors through $M$ and admits at least two fixed points.
\end{enumerate}
We may view this result as a condition on the Hom spaces Hom$(L, M)$ and Hom$(M, L)$. Note that if $S$ and $T$ are simple modules, then there is an isomorphism $S\rightarrow T$ if and only if Hom$(S, T)$ is different from zero. We also show (Proposition~\ref{aleph0}) that cyclic uniserial modules may have few endomorphisms. Finally we use a nice example due to Osofsky \cite{O} to construct a non cyclic uniserial module with countably many non cyclic submodules (Example~\ref{Osofsky}).

Throughout this paper $K$ will be an arbitrary field and modules will be left modules. Moreover, Soc$M$ will be the socle of any module $M$. A module $M$ is said to be \emph{uniserial} if its submodule lattice is a chain. As usual we say that a module $M$ is \emph{uniform} if the intersection of two nonzero submodules of $M$ is different from zero.

Let $x$ be a vertex of a quiver $Q$. Then $S(x)$ will denote the
simple representation corresponding to the vertex $x$. On the
other hand, $P(x)$ (resp. $I(x)$) will denote the indecomposable
projective (resp. injective) representation corresponding to the
vertex $x$. Sometimes, for short, $S(x)$ is replaced by $x$. As in \cite{S},
pictures of the form
$$
\begin{array}{rlc}
 1&&2\\
 &3
\end{array}\quad,
\begin{array}{rlc}
 1&&2\\
 &2
\end{array}\quad,
\begin{array}{rlc}
 1\\
 2
\end{array}\quad,
\begin{array}{rlc}
 2\\
 2
\end{array}\quad,
\begin{array}{rlc}
 1\\
 3
\end{array}\cdots
$$
denote the composition series of indecomposable modules. Our convention for the composition of paths $p, q$ in the path algebra is as in \cite{ARS}, namely $qp$ stands for {\it $q$ after p} whenever the concatenation is defined. For more background on quivers we refer to \cite{ARS} and \cite{S}.

\section{Results}

We start with an easy observation.

\begin{lem} \label{inj}
Let $R$ be any ring and $L$ and $M$  uniform $R$-modules. Let $f:L\rightarrow M$ and $g:M\rightarrow L$ be morphisms such that $(g\circ f)(x)=x$ for some nonzero element $x\in L$. Then $f$ and $g$ are injective.
\end{lem}

\begin{proof}
We have that $(g\circ f)(rx)=rx$ for every $r\in R$. This implies that
\begin{center}
 $Rx\cap\mbox{Ker}f=0~\mbox{and}~f(Rx)\cap\mbox{Ker}g=0$
\end{center}
Since $Rx$ and $f(Rx)$ are nonzero, then Ker$f=0$ and Ker$g=0$ ($L$ and $M$ are uniform). The lemma is proved.
\end{proof}

Under certain assumptions on $L$ and $M$ (for instance when $L$ and $M$ are injective \cite{B}, and more generally when the Schr\"{o}der-Bernstein problem has a positive solution \cite{AKS}), the hypotheses of Lemma~\ref{inj} imply that $L$ and $M$ are isomorphic.

Now we are giving our first main result:

\begin{thm} \label{res1} Let $L$ and $M$ be uniform modules of finite length $n$ over a ring $R$. The following are equivalent:
\begin{enumerate}
\item[(i)] $L$ is isomorphic to $M$;
\item[(ii)] There exist two morphisms $f: L\to M$, $g: M\to L$ and a nonzero element $x\in L$ such that $(g\circ f)(x)=x$;
\item[(iii)] There exist $n$ morphisms $f_1, \ldots , f_n$ with $f_1:L\to M$, $f_2: M\to L, \ldots $ such that $f_n\circ \ldots \circ f_1\neq 0$.
\end{enumerate}
\end{thm}

\begin{proof} (i)$\Rightarrow$(ii) This is obvious.

\indent (ii)$\Rightarrow$(iii) Let $f$ and $g$ be as in (ii), and let $h=(g\circ f)^{n/2}$ if $n$ is even, and $h=f\circ (g\circ f)^{\frac{n-1}{2}}$ if $n$ is odd. Then $h\neq 0$ and $h$ is of the form $f_n\circ \ldots \circ f_1$ with $f_j=f$ if $j$ is odd and $f_j=g$ if $j$ is even. Hence (iii) holds.

\indent (iii)$\Rightarrow$(i) Assume that $f_1, \ldots , f_n$ satisfy (iii) and that $L$ is not isomorphic to $M$. Then $f_1(L)$ has length $\leq n-1$. Indeed,
if $f_1(L)$ has length $n$, then we have $f_1(L)=M$. So, $L/{\rm Ker}f_1\cong M$ implies that $L/{\rm Ker}f_1$ has length $n$. This means that Ker$f_1$ has length $0$, namely Ker$f_1=0$. Hence $L\cong M$, a contradiction.
On the other hand, $(f_2\circ f_1)(L)$ has length $\leq n-2$. Indeed, since $M$ is uniform, Ker$f_2\cap f_1(L)$ is nonzero. So, $f_1(L)/({\rm Ker}f_2\cap f_1(L))\cong (f_2 \circ f_1)(L)$ implies that $(f_2 \circ f_1)(L)$ has length $\leq n-2$.
Proceeding by induction, we conclude that $(f_{n-1}\circ \ldots \circ f_2\circ f_1)(L)$ has length $\leq 1$. Since Ker$f_n\neq 0$, we obtain $f_n\circ \ldots \circ f_1=0$, a contradiction to the hypothesis. Hence the result holds.
\end{proof}

Since any uniserial module is uniform, we have the following corollary.

\begin{cor} \label{res2} Let $L$ and $M$ be uniserial modules of finite length $n$ over a ring $R$. The following are equivalent:
\begin{enumerate}
\item[(i)] $L$ is isomorphic to $M$;
\item[(ii)] There exist two morphisms $f: L\to M$, $g: M\to L$ and a nonzero element $x\in L$ such that $(g\circ f)(x)=x$;
\item[(iii)] There exist $n$ morphisms $f_1, \ldots , f_n$ with $f_1:L\to M$, $f_2: M\to L, \ldots $ such that $f_n\circ \ldots \circ f_1\neq 0$.
\end{enumerate}
\end{cor}

\begin{rem}\rm
Note that in Theorem~\ref{res1}, while proving $(i)\Rightarrow (ii)\Rightarrow (iii)$, we are not using the ``uniform" condition on $M$ and $L$; and we cannot remove the hypotheses that $L$ and $M$ have the same finite length, or that they both uniform while proving $(iii)\Rightarrow (i)$. Indeed let $A$ be the $K$-algebra given by a quiver with one vertex, say 1, and a loop $a$ around 1 satisfying $a^{2}=0$. With the usual conventions let $f:\begin{array}{rlc}  1\\  1 \end{array}\twoheadrightarrow 1$ and $g:1\hookrightarrow \begin{array}{rlc}  1\\  1 \end{array}$ be the obvious morphisms. Then $g\circ f\neq 0$, but $\begin{array}{rlc}  1\\  1 \end{array}$ and $1$ are uniserial of length $\leq 2$. Now let $A$ be the $K$-algebra given by the quiver
$$
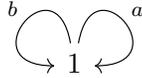
\begin{tikzcd}
1 \arrow[out=80,in=0,loop, "a"] \arrow[out=100,in=180,loop, swap, "b"]
\end{tikzcd}$$
with all paths of length two equal to zero. Let $L$ and $M$ be the indecomposable projective or injective modules described by the following pictures in an obvious way.
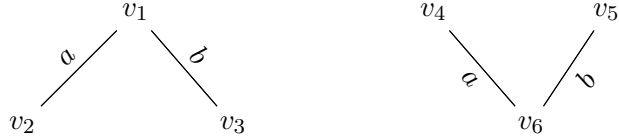
\begin{center}
\begin{tikzpicture}
    \node at (0,0) (nodeA) {$v_2$};
    \node at (1.5,1.5) (nodeB) {$v_1$};
    \node at (2.8,0) (nodeC) {$v_3$};

    \draw (nodeA) -- (nodeB) -- (nodeC);

    \draw (nodeA) -- (nodeB) node [midway, above, sloped] (EdgeAB) {$a$};
    \draw (nodeB) -- (nodeC) node [midway, above, sloped] (EdgeBC) {$b$};
\end{tikzpicture}
~~~~~~~~~~~~~~
\begin{tikzpicture}
    \node at (1.5,1.5) (nodeB) {$v_4$};
    \node at (2.8,0) (nodeC) {$v_6$};
    \node at (3.8,1.5) (nodeD) {$v_5$};

    \draw (nodeB) -- (nodeC) -- (nodeD);

    \draw (nodeB) -- (nodeC) node [midway, below, sloped] (EdgeBC) {$a$};
    \draw (nodeC) -- (nodeD) node [midway, below, sloped] (EdgeCD) {$b$};
\end{tikzpicture}
\end{center}
Let $f:L\rightarrow M$ be the morphism such that $f(v_{1})=v_{4}$. Then we clearly have $f(v_{2})=v_{6}$. Next let $g:M\rightarrow L$ be the morphism such that $g(v_{4})=v_{2}$ and $g(v_{5})=0$. Then we have $(f\circ g\circ f)(v_{1})=(f\circ g)(v_{4})=f(v_{2})=v_{6}$, and so $f\circ g\circ f\neq 0$.
\end{rem}

The next example shows that we cannot replace (iii) in Corollary \ref{res2} by a similar condition on $n-1$ maps even if the two uniserial modules are projective-injective.

\begin{exam}\label{ex1}\rm There are a finite dimensional $K$-algebra $A$ and two uniserial non-isomorphic projective-injective modules $P$ and $Q$ with the following properties:
\begin{enumerate}
\item[(a)] $P$ and $Q$ have dimensions and lengths equal to $3$;
\item[(b)] ${\rm Hom}_A(P, Q)$ (resp. ${\rm Hom}_A(Q, P)$) is generated by an element $f$ (resp. $g$) such that $g\circ f\neq 0$ (resp. $f\circ g\neq 0$).
\end{enumerate}

\medskip \noindent
{\bf Construction:} Let $A$ be the $K$-algebra given by the quiver
$1 \stackrel [b]{a}{\rightleftarrows} 2$
 with relations $aba=0$ and $bab=0$. Then the projective-injective modules
$P=\begin{array}{rlc}
 1\\
 2\\
 1
\end{array}, \quad Q=\begin{array}{rlc}
 2\\
 1\\
 2
\end{array}$ satisfy (a) and we have $dim{\rm Hom}_A(P, Q)=1= dim{\rm Hom}_A(Q, P)$. Moreover the morphisms $f$ and $g$ in (b) have the property that $(g\circ f)(P)={\rm Soc}P$, $(f\circ g)(Q)={\rm Soc}Q$, $f\circ g\circ f=0$ and $g\circ f\circ g=0$.
\end{exam}

\begin{prop} \label{LM}
Let $L$ and $M$ be uniserial $R$-modules such that the lattice of their submodules is of the following form.
\begin{align}
& \bullet \notag \\
& \vdots \notag \\
& \vdots \notag \\
& \bullet \notag \\
& \mid \notag \\
& \bullet \notag \\
& \mid \notag \\
& \bullet \notag
\end{align}
The following conditions are equivalent:
\begin{enumerate}
  \item [(i)] $L$ is isomorphic to $M$.
  \item [(ii)] There exist a nonzero element $x\in L$ and two morphisms $f:L\rightarrow M$ and $g:M\rightarrow L$ such that $(g\circ f)(x)=x$.
\end{enumerate}
\end{prop}

\begin{proof}
$(i)\Rightarrow(ii)$ This is obvious.\\
$(ii)\Rightarrow(i)$ We first deduce from Lemma~\ref{inj} that $f$ is injective. Consequently $f(L)$ is a submodule of $M$ which is not of finite length. It follows that $f$ is surjective. Hence $(i)$ holds.
\end{proof}
\begin{exam}
\rm There are uniserial modules $L$ and $M$ and morphisms $f:L\rightarrow M$ and $g:M\rightarrow L$ such that $L\ncong M$, but $(g\circ f)^{n}\neq 0$ for any $n\geq 1$.
\end{exam}

\medskip\noindent
{\bf Construction:} Let $A$ be the $K$-algebra given by the quiver $1 \stackrel [b]{a}{\rightleftarrows} 2$, and let $L$ and $M$ the modules
$$I(1)=\begin{array}{c}
    \vdots \\
    2 \\
    1 \\
    2 \\
    1
  \end{array}~~~~\mbox{and}~~~~I(2)=\begin{array}{c}
    \vdots \\
    1 \\
    2 \\
    1 \\
    2
  \end{array}~~~~\mbox{respectively}.
$$
Next let $f:I(1)\rightarrow I(2)$ and $g:I(2)\rightarrow I(1)$ be morphisms with simple kernel. Then $(g\circ f)^{n}$ is surjective for any $n\geq 1$, but $I(1)\ncong I(2)$.

\bigskip
As observed in the introduction, the sequence of the composition factors of a uniserial module $U$ of finite length does not determine $U$ up to isomorphism. More generally, a similar result holds for the largest subquotients of a uniserial module of finite length.

\begin{exam} \label{subq} \rm
The subquotients of length $n-1\geq 2$ of a uniserial module $U$ of finite length $n$ do not determine $U$ up to isomorphism.
\end{exam}

\medskip \noindent
{\bf Construction:} Let $A$ be the $K$-algebra given by the Euclidean diagram $\tilde{A_{n}}$ with the following orientation.
$$
\begin{tikzcd}
  1 \ar[r,"a_1"]
    \ar[rrrrr,bend right=50,looseness=0.3,"b"']
& 2 \ar[r,"a_2"]
& 3 \ar[r,"a_3"]
& \cdots
  \ar[r,"a_{n-2}"] 
& n-1 \ar[r,"a_{n-1}"]
& n
\end{tikzcd}
$$
Let $L=P(1)/(a_{n-1}\cdots a_{2}a_{1}-b)$ and let $M=P(1)/(b)$. Then $L$ and $M$ are uniserial $A$-modules of length $n$. Since $bL\neq 0$ and $bM=0$, we have $L\ncong M$. On the other hand $L$ and $M$ have a maximal submodule isomorhic to $P(2)$ and a maximal factor module isomorhic to $I(n-1)$.

\begin{rem} \rm
We know from \cite[Introduction]{FS} and \cite[Proposition 2]{Facchini} that the subquotients of finite length of a uniserial module behave quite differently in the commutative and in the noncommutative case. Hence it is natural to wonder if the situation described in Example~\ref{subq} happens also in the commutative case. To see that this happens also in the commutative case, let $A$ be the $K$-algebra given by the quiver
$$\begin{tikzcd}
1 \arrow[out=80,in=0,loop, "a"] \arrow[out=100,in=180,loop, swap, "b"]
\end{tikzcd}$$
with relations $a^{n}=0, b^{2}=0, ab=0$ and $ba=0$. Then the following pictures
$$\begin{tikzcd}
  v_{1} \ar[r,"a"]
    \ar[rrr,bend right=50,looseness=0.3,"b"']
& \cdots \ar[r,"a"]
& v_{n-1} \ar[r,"a"]
& v_{n}
\end{tikzcd}~~~~~~~~
\begin{tikzcd}
  w_{1} \ar[r,"a"]
  &\cdots \ar[r,"a"]
  & w_{n-1} \ar[r,"a"]
  & w_{n}
\end{tikzcd}
$$
describe two non isomorphic uniserial modules of length $n$, say $V$ and $W$ (with bases $v_{1}, v_{2},\ldots, v_{n}$ and $w_{1}, w_{2},\ldots, w_{n}$ respectively) such that the subquotients of $V$ and $W$ of length $n-1$ are of the following form.
$$\begin{tikzcd}
  \bullet \ar[r, "a"]
& \bullet
\cdots \ar[r, "a"]
& \bullet \ar[r,"a"]
& \bullet
\end{tikzcd}$$
\end{rem}

As we shall see, also the dual of Proposition~\ref{LM} holds.
\begin{prop} \label{chain}
Let $L$ and $M$ be uniserial modules such that the lattice of their submodules is of the following form.
\begin{align}
& \bullet \notag \\
&  \mid \notag \\
& \bullet \notag \\
&  \mid \notag \\
& \bullet \notag \\
& \vdots \notag \\
& \vdots \notag \\
& \vdots \notag \\
& \bullet \notag
\end{align}
The following conditions are equivalent:
\begin{enumerate}
  \item [(i)] $L$ is isomorphic to $M$.
  \item [(ii)] There exist a nonzero element $x\in L$ and two morphisms $f:L\rightarrow M$ and $g:M\rightarrow L$ such that $(g\circ f)(x)=x$.
\end{enumerate}
\end{prop}

\begin{proof}
$(i)\Rightarrow(ii)$ This is obvious.\\
$(ii)\Rightarrow(i)$ Assume $L=L_{0}\supset L_{1}\supset L_{2}\supset \cdots$ and $M=M_{0}\supset M_{1}\supset M_{2}\supset \cdots$ are the sequences of all nonzero submodules of $L$ and $M$.
Then there exist $i$ and $j$ such that
\begin{itemize}
        \item [(1)] $Rx=L_{i}$ and $f(L)=M_{j}$.\\

Since Lemma~\ref{inj} guarantees that $f$ and $g$ are injective, we have
        \item [(2)] $f(L_{n})\subseteq M_{j+n}$ and $g(M_{n})\subseteq L_{n}$ for any $n$.\\

It follows that
        \item [(3)] $L_{i}=(g\circ f)(L_{i})=g(f(L_{i}))\subseteq g(M_{j+i})\subseteq L_{j+i}$.
\end{itemize}
Consequently we have $j=0$, and so $f:L\rightarrow M$ is an isomorphism.
\end{proof}
We can now give the second main result of our paper.

\begin{thm} \label{fixpoint}
Let $L$ and $M$ be uniserial modules such that the lattice of their submodules is isomorphic to a sublattice of~~$\mathbb{Z}\cup\{+\infty, -\infty\}$ with the usual order. Then $L$ is isomorphic to $M$ if and only if there is an endomorphism $h$ of $L$ with the following properties:
\begin{itemize}
  \item [(a)] $h$ factors through $M$.
  \item [(b)] $h(x)=x$ for some nonzero $x\in L$.
\end{itemize}
\end{thm}

\begin{proof}
If the lattice of the submodules of $L$ and $M$ is isomorphic to either a finite interval of $\mathbb{N}$, or to $\mathbb{N}\cup\{+\infty\}$, or to $(\mathbb{Z}\setminus \mathbb{N})\cup \{-\infty\}$, then the claim follows from Theorem~\ref{res1} and Propositions \ref{LM} and \ref{chain}. Assume now that this lattice is $\mathbb{Z}\cup \{+\infty, -\infty\}$. Lef $f:L\rightarrow M$ and $g:M\rightarrow L$ be morphisms such that $h=g\circ f$. Then Lemma~\ref{inj} implies that $f$ and $g$ are injective. Hence we have $0\neq f(L)\subseteq M$ and $f(L)$ does not have a maximal submodule. Consequently $f(L)=M$, and so $f$ is an isomorphism.
\end{proof}

To see that there are uniserial non cyclic modules with few endomorphisms and many submodules, it is enough to consider the following example.\\

\begin{exam}
\rm There is a uniserial module $W$ with the following properties:
\begin{itemize}
  \item [(a)] $W$ is not cyclic and $\mbox{End}~W\cong K$;
  \item [(b)] The lattice of submodules of $W$ is isomorphic to $\mathbb{Z}\cup \{+\infty, -\infty\}$
\end{itemize}
\end{exam}
\medskip \noindent
{\bf Construction:} Let $A$ be the $K$-algebra given by a quiver with one vertex and countably many loops $\alpha_{i}$ with $i\in \mathbb{Z}$. Next let $W$ be the $A$-module described by the following picture.
$$
\begin{tikzcd}
& \cdots \ar[r, "\alpha_{-2}"]
& v_{-1} \ar[r, "\alpha_{-1}"]
& v_{0} \ar[r, "\alpha_{0}"]
& v_{1} \ar[r, "\alpha_{1}"]
& v_{2} \ar[r, "\alpha_{2}"]
& \cdots
\end{tikzcd}
$$
Then $W$ satisfies (a) and (b).\\

The next example shows that the composition factors of a uniserial module (over a non commutative ring) may be vector spaces of different dimension.

\begin{exam} \rm
Let $C$ and $D$ be fields with $C\subseteq D$. Then there is a $C$-algebra $R$ with the following properties:
\begin{enumerate}
  \item [(i)] $R$ is a left artinian ring.
  \item [(ii)] $R$ is an artin $C$-algebra if and only if $[D:C]$ is finite.
  \item [(iii)] The left ideals of $R$ which are simple $R$-modules are all isomorphic, projective and parametrized by the projective line over $D$.
  \item [(iv)] $R$ admits a uniserial module $P$ of length two, such that
  $$dim_{C}~\mbox{Soc}P=[D:C]~~~~\mbox{and}~~~~dim_{C}~P/\mbox{Soc}P=1$$
\end{enumerate}
\medskip \noindent
{\bf Construction:} Let $R$ be the $C$-algebra
$\left(\begin{array}{cc}
                D & D \\
                0 & C \\
              \end{array}
            \right)$.
Then $(i)$ follows from the fact that any proper nonzero left ideal of $R$ belongs to the following list:
$$~~~~~~~~~~~~~~~~~~~~~~~~~~~~~~~~~~~~I_{k}=R\left(\begin{array}{cc}
                  1 & k \\
                  0 & 0 \\\end{array}\right),~~~~~~~~~~~~~~~~~~~~~~
                  I_{\infty}=R\left(\begin{array}{cc} 0 & 1 \\
                                                      0 & 0 \\
                                    \end{array}\right)=\left(\begin{array}{cc}  0 & D \\
                                                                                0 & 0 \\
                                                             \end{array}\right)=J(R),$$
                 $$I_{0}\oplus I_{\infty}=\left(\begin{array}{cc} D & D \\
                                           0 & 0 \\
                        \end{array}\right),~~~~~~~~~~~~~~
                  R\left(\begin{array}{cc}  0 & 0 \\
                                            0 & 1 \\
                         \end{array}\right)=\left(\begin{array}{cc} 0 & D \\
                                                                    0 & C \\
                                                  \end{array}\right)$$
Then $I_{\lambda}$ for $\lambda\in D\cup \{\infty\}$ is a simple projective $R$-module. Indeed $I_{0}$ is a summand of $_{R}R$ and the map $I_{0}\rightarrow I_{k}$ (respectively $I_{0}\rightarrow I_{\infty}$) such that
$$\left(\begin{array}{cc}
                  x & 0 \\
                  0 & 0 \\\end{array}\right)\mapsto \left(\begin{array}{cc} x & kx \\
                                                                            0 & 0 \\
                                                          \end{array}\right)~~(\mbox{respectively}~~
\left(\begin{array}{cc}
                  x & 0 \\
                  0 & 0 \\\end{array}\right)\mapsto \left(\begin{array}{cc} 0 & x \\
                                                                            0 & 0 \\
                                                          \end{array}\right))$$
is an isomorphism of $R$-modules. Hence $(iii)$ holds. Since $dim_{C}R$ is finite if and only if $[D:C]$ is finite, also $(ii)$ holds. Finally let
$P=\left(\begin{array}{cc} 0 & D \\
                           0 & C \\
         \end{array}\right)$. Then we have Soc$P=\left(\begin{array}{cc}  0 & D \\ 0 & 0 \\ \end{array}\right)$. Consequently $dim_{C}$~Soc$P=[D:C]$ and $dim_{C}~P/$Soc$P=1$, as claimed in $(iv)$.
\end{exam}

As we shall see, uniserial modules over a $K$-algebra may have a small endomorphism ring.\\

\begin{prop}\label{aleph0}
There exist $K$-algebras $A$ and uniserial $A$-modules $U$ such that one of the following conditions hold:
\begin{itemize}
  \item [(i)] ${\rm End}_{A}~U\cong K$~~~~and~~~~$1\leq \mbox{dim}_{K}~U\leq \aleph_{0}$;
  \item [(ii)] ${\rm End}_{A}~U\cong K[x]/(x^{2})$~~~~and~~~~$2\leq \mbox{dim}_{K}~U\leq \aleph_{0}$.
\end{itemize}
\end{prop}

\begin{proof}
\rm (i) For any $n\geq 1$ let $A$ be the $K$-algebra given by the Dynkin diagram
$$
\begin{tikzcd}
  1 \ar[r]
& 2 \ar[r]
&  \ar[r]
\cdots \ar[r] & n
\end{tikzcd}
.$$
Then for any $j=1, 2, \ldots n$, the uniserial injective module $I(j)$ has dimension $j$ and endomorphism ring $K$.
Next let $A$ be the $K$-algebra given by a quiver with one vertex and countably many loops $\alpha_{1}, \alpha_{2}, \alpha_{3}, \cdots$. Finally let $U$ be the uniserial $A$-module described by the following picture.
$$
\begin{tikzcd}
v_{1} \ar[r,"\alpha_1"]
& v_{2} \ar[r,"\alpha_2"]
& v_{3} \ar[r,"\alpha_3"]
& \cdots\cdots
\end{tikzcd}
$$
Then we have $U=Av_{1}$ and $\alpha_{i}v_{1}=0$ for any $i>1$. Consequently $h(v_{1})\in <v_{1}>$ for any $h\in~\mbox{End}_{A}~U$. Hence $\mbox{End}_{A}~U\cong K$ and so $(i)$ holds.\\
\rm (ii) Let $A$ be the $K$-algebra given by one of the following quivers.
$$
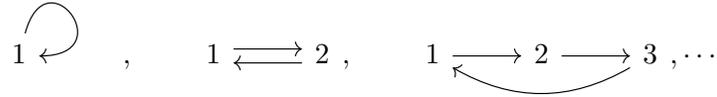
\begin{tikzcd}
1 \arrow[out=75,in=0,loop]
\end{tikzcd},~~~~~~ \begin{tikzcd}1 \arrow[r,shift left]  & 2 \arrow[l,shift left] \end{tikzcd},~~~~~~
\begin{tikzcd} 1 \ar[r]  & 2 \ar[r] & 3 \ar[ll,bend left]\end{tikzcd},\cdots
$$
Then the uniserial modules $\begin{array}{rlc}
 1\\
 1
\end{array}, \begin{array}{rlc}
 1\\
 2\\
 1
\end{array}, \begin{array}{rlc}
 1\\
 2\\
 3\\1
\end{array}, \cdots$ have dimension $2, 3, 4,\cdots$ and endomorphism ring $K[x]/(x^{2})$. To end the proof, let $V$ be the cyclic artinian $A$-module constructed by Osofsky in \cite{O}. Then $V$ is a $K$-vector space with basis $v_{0}, v_{1}, v_{2}, \cdots$ and $A$ is the subalgebra of $\mbox{End}_{K}~V$ generated by the linear maps $f_{0}, f_{1}, f_{2}, \cdots $ described by the following picture.
$$
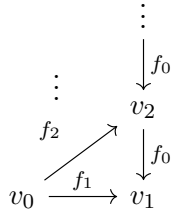
\begin{tikzcd}
&\vdots\arrow{d}{f_{0}} \\ & v_{2} \arrow{d}{f_{0}} \\
v_{0} \arrow{ur}{\begin{array}{rlc} \vdots\\f_{2}\end{array}} \arrow{r}{f_{1}} & v_{1}
\end{tikzcd}
$$
Thus we have $V=Av_{0}$ and Ker$f_{0}=<v_{0}, v_{1}>$. Consequently $f_{1}\in \mbox{End}_{A}V$. Since $f^{2}_{1}=0$, we obtain $\mbox{End}_{A}V=K[f_{1}]\simeq K[x]/(x^{2})$. Hence also $(ii)$ holds.
\end{proof}

We see that by ``glueing together" non isomorphic copies of the Osofsky's module in \cite{O}, we obtain a more complicated module in the following example.

\begin{exam}\label{Osofsky}
There are a $K$-algebra $A$ and a uniserial $A$-module $U$ with the following properties:
\begin{itemize}
  \item [(a)] $U$ is the union of cyclic submodules $L_{n}$ such that $L_{1}\subset L_{2}\subset L_{3}\subset \cdots$;
  \item [(b)] $U$ is the union of non cyclic submodules $M_{n}$ such that $M_{1}\subset M_{2}\subset M_{3}\subset \cdots$.
\end{itemize}
\end{exam}
\medskip \noindent
{\bf Construction:} Let $U$ be a $K$-vector space with a basis of the form
\begin{align}
v_{11}, v_{12},\ldots, v_{1\infty}, v_{21}, v_{22},\ldots, v_{2\infty}, v_{31}, \dots \label{v}
\end{align}
Let $A$ be the subalgebra of $\mbox{End}_{K}U$ generated by the linear maps $f_{mn}, g_{m}, h_{m}$ described by the following picture.
\begin{center}
\begin{tikzcd}
 & \vdots \arrow[d, "g_{3}"] &  & \vdots  \arrow[d, "g_2"] &  &\vdots \arrow[d, "g_1"] \\
 & v_{32} \arrow[d, "g_{3}"] &  & v_{22} \arrow[d, "g_2"] &  & v_{12} \arrow[d, "g_1"] \\
 \cdots \arrow[ru, "\vdots"] \arrow[r, "f_{31}"'] & v_{31} \arrow[r, "h_2"'] & v_{2\infty} \arrow[r, "f_{21}"'] \arrow[ru, "\overset{\vdots}{f_{22}}"] & v_{21}
 \arrow[r, "h_1"'] & v_{1\infty} \arrow[ru, "\overset{\vdots}f_{12}"] \arrow[r, "f_{11}"'] & v_{11}
\end{tikzcd}
\end{center}
With this notation, let
$$L_{n}=Av_{n\infty}~~\mbox{and}~~M_{n}=\bigcup_{i\neq\infty}Av_{ni}~~\mbox{for any}~~n.$$
Then the $L_{n}$'s satisfy $(a)$ and the $M_{n}$'s satisfy $(b)$. Moreover we clearly have
\begin{align}
Av_{11}\subset Av_{12}\subset\ldots  M_{1}\subset L_{1}\subset Av_{21}\subset Av_{22}\subset\ldots M_{2}\subset L_{2}\subset Av_{31}\subset Av_{32}\subset \ldots \label{R}
\end{align}
Let $v=av_{ij}+w$ where $0\neq a\in K$ and $w$ is a linear combination of vectors on the left of $v_{ij}$ in (\ref{v}). If either $j=1$ or $j=\infty$, then any vector on the left of $v_{ij}$ in (\ref{v}) belongs to $Av$. Hence we have $Av=Av_{ij}$. Assume now $1<j<\infty$. Then we have $g^{j-1}_{i}(v), g^{j-2}_{i}(v), \ldots, g_{i}(v)\in Av$. This implies that $v_{i1}, v_{i2}, \ldots, v_{i(j-1)}\in Av$. Hence we have $Av=Av_{ij}$. Consequently any proper submodule of $U$ appears in (\ref{R}), and so $U$ is uniserial, as claimed.\\

\end{document}